\documentclass{amsart}
\usepackage{amsfonts,amssymb,amscd,amsmath,amsthm,enumerate,verbatim,calc,graphicx,color, epsf}

\headsep=1cm \numberwithin{equation}{section}
\newtheorem{thm}{Theorem}[section]

\newtheorem{defn}[thm]{Definition}

\newtheorem{exam}[thm]{Example}
\newtheorem{rem}[thm]{Remark}

\newcommand{\ini}{\mbox{in}\,}

\renewcommand{\dim}{\mbox{dim}\,}

\newcommand{\reg}{\mbox{reg}}

\begin{document}

\bibliographystyle{amsplain}

\title[Ideals with large Regularity and Regularity Jumps]{Some Remarks on Ideals with large Regularity and Regularity Jumps}

\author{Keivan Borna}
\address{ Keivan Borna \\Faculty of Mathematics and Computer Science,
Kharazmi University, Tehran, Iran}
\email{borna@khu.ac.ir}

\author{Abolfazl Mohajer}
\address{Abolfazl Mohajer\\Johannes Gutenberg-Universit\"{a}t Mainz,  Staudingerweg 9, D 55128 Mainz, Germany}
\email{mohajer@uni-mainz.de}





\date{Aug.-17-2015}


\begin{abstract}
This paper exhibits some new examples of the behavior of the
Castelnuovo-Mumford regularity of homogeneous ideals in polynomial
rings. More precisely, we present new examples of homogenous ideals with large regularity compared to the generating degree. Then we consider the regularity jumps of ideals. In particular we provide an infinite family of ideals having regularity jumps at a certain power.
\end{abstract}

\maketitle

\section{Introduction}

Castelnuovo-Mumford regularity, or simply regularity, together
with the projective dimension are the most important invariants of
a homogenous ideal in a polynomial ring $k[x_1,...,x_n]$ (or a
closed subscheme of $\mathbb{P}^{n}$). It measures the extent of
cohomological complexity of such an ideal. Explicitly, the
regularity is a measure for the Hilbert function of the ideal, or
the ideal sheaf, to become polynomial; see \cite{E}, \S 4. Bayer and
Mumford \cite{BM} point out that the regularity can also be considered
as a measure of the complexity of computing the Gr\"obner bases.
More generally, let $S=k[x_1,...,x_n]$ with $k$ a field of
characteristic zero and $M$ be a finitely generated graded
$S$-module. Consider a minimal
graded free resolution of $M$ as follows.
\[
\mathbb{F}: ...\to F_i\xrightarrow{\delta_i} F_{i-1}\xrightarrow{\delta_{i-1}}...\to F_0\xrightarrow{\delta_0} M
\]
There exists integers $a_{ij}$ such that $F_i=\sum S(-a_{ij})$.
The regularity of $M$, denoted $reg(M)$, is then defined to be the
supremum of the numbers $a_{ij}-i$. For $d\geq reg(M)+1$, the Hilbert function $H_M(d)$
agrees with the Hilbert polynomial $P_M(d)$.

Another way of defining the regularity is through graded local
cohomology modules $H^{i}_{m}(M)$ for each $0\leq i\leq \dim(M)$, where $m=(x_1,...,x_n)$ denotes
the irrelevant maximal ideal of $S$. As such modules are Aritinan,
one can define $end(H^{i}_{m}(M))$ as the maximum integer $k$ such
that $H^{i}_{m}(M)_k\neq 0$. Then one can equivalently define
\[
reg(M)=max \{end(H^{i}_{m}(M))+i\}
\]

For equivalent definitions and various algebro-geometric
properties of the regularity we refer to \cite{E, BS, L}.

In the case that $M=I$ is a homogenous ideal in $S$, we remark that:

\begin{rem} Let $I$ be a homogenous ideal in the
polynomial ring $S=k[x_1,...,x_n]$ and $m$ be the irrelevant
maximal ideal of $S$. If $I$ is not $m$-primary, that is, if
$\sqrt{I}\neq m$, then $reg(I)=min\{\mu|H^{i}(S/I)_{\mu-i}=0\; 
\forall i\}$; see \cite{Ch}, Proposition 9.5.
\end{rem}

If $I$ is a homogenous ideal generated in a single degree $d$,
then $reg(I)\geq d$. One important problem in studying the
Castelnuovo-Mumford regularity of ideals is to find ideals whose
regularity if large relative to the generating degree. Mayr and
Meyer \cite{MM} have given examples of ideals in polynomial rings in
$10n+2$ variables whose regularity is a doubly exponential
function of $n$ and polynomial in the generating degree $d$; see
\cite{MM}. Caviglia \cite{C} was probably the first to produce an ideal in a
polynomial ring with fixed number of variables and three
generators whose regularity is much larger than the generating
degree. There have been other attempts to find examples of ideals
with large regularity; see for example
\cite{BMNSSS}.

Another interesting problem is to consider the regularity of
powers of an ideal $I$. In \cite{Co} an interesting notion, namely that
of \emph{regularity jumps} has been defined. An ideal has
regularity jump at the $k$-th power if $reg(I^k)-reg(I^{k-1})>d$.
In the same article the author mentions many new and known
examples of ideals with this property. In \cite{B} the author presented a simple
criterion in terms of Rees algebra of a specific ideal to show that
high enough powers of certain ideals have linear resolution.

Our aims in this paper are two folds. First we present new results of homogenous
ideals with large regularity comparing to their generating degree. Then we focus on
the regularity jumps of ideals and provide an
infinite family of ideals having regularity jumps at a certain power.

This paper is structured as follows. In the first section we
discuss several variants of Caviglia's example and give further
examples of ideals with stronger regularities. In particular, we
explain (see Remark 2.7) why we expect that a generalization of
our example would produce polynomially large regularities of
arbitrary degree. In the second section we consider the problem of
ideals with regularity jumps and show that an infinite family of
ideal $I_n$ for $n\geq3$ have regularity jump at $k=2$. The ideals
$I_n$ define Cohen-Macaulay rings of minimal multiplicity
indicating that even among such ideals one can find examples whose
squares do not have linear resolution. The ideal $I_3$ has been
shown in \cite{Co} to have such a regularity jump by declaring the
existence of a non-linear second syzygy. Our contribution here is
to show that for all $n\geq 3$ the ideal $I_n^2$ has regularity
strictly greater than $4$. We achieve this by local cohomological
methods.

\section{Ideals with large regularity}

In this section we are going to construct homogenous ideals with
large Castelnuovo-Mumford regularity. Of course, by large
regularity we mean that $\frac{reg(I)}{d(I)}$ could be made
arbitrarily large, where $d=d(I)$ is the degree of generators of
$I$. Note that we only consider \emph{equigenerated} ideals, i.e.,
homogenous ideals all of whose generators are of the same degree.
As it is mentioned in \cite{E}, there are only few known examples in
small number of variables of such ideals and as the author in \cite{E}
mentions it is interesting to construct more such ideals
especially with fix number of variables. \\

Our main tool to produce ideals with large regularity is the
notion of weakly stability which was developed in \cite{C}. In fact, we
also compute the initial ideal of the ideals we consider. Although
this imposes lengthy computations even in relatively simple cases,
this approach has the advantage of demonstrating the Bayer-Mumford
philosophy in \cite{BM} which we mentioned in the introduction that
\emph{regularity is a measure of complexity of computing the
Gr\"obner basis} and hence ideals with larger regularity give rise
to more complicated initial ideals. Let us also remark that whereas
the following ideals with large regularity are not prime ideals, it
is expected that for prime ideals much smaller upper bounds should
exist. See \cite{EG}. We first recall the definition and some properties
that we will need later.\\


\begin{defn}
A monomial ideal $(u_1,...,u_n)$ is
called \emph{weakly stable} if for each generator $u_i$, there
exist $a_j\in \mathbb{N}$ such that
$x_{j}^{a_j}(\frac{u_i}{x^{\infty}_{m(u_i)}})\in I$ for every
$j<m(u_i)$. Where $m(u_i)$ is the maximum of all $j$ such that
$x_j$ divides $u_i$ and $x^{\infty}_{m(u_i)}$ is the highest
power of $x_{m(u_i)}$ dividing the monomial $u_i$.
\end{defn}

Perhaps the best known and simplest example of ideals with large
regularity in a fixed polynomial ring was given by Caviglia in
\cite{C}. This is the ideal $I=(x_1^d,x_2^d,x_1x_3^{d-1}-x_2x_4^{d-1})$ for which
$reg(I)=d^2-1$. Here we first investigate variants of this ideal
to obtain more examples of ideals with large regularity. Let us
recall two results from \cite{C} that relate the regularity of an ideal
to that of its initial ideal:

\begin{thm} Let $I\subseteq k[x_1,...,x_n]$ be a weakly
stable ideal generated by the minimal system $u_1,...,u_r$. Assume that $u_1>u_2>...>u_r$ with respect to the reverse lexicogeraphical order (which is different from revlex). Then $reg(I)=max \{deg(u_i)+C(u_i)\}$, where $C(u_i)$ is the highest degree of a monomial $\nu$ in
$k[x_1,...,x_j]$ such that $\nu \notin
((u_1,...,u_{i-1}):u_i)$\\
\end{thm}

\begin{thm} Let $I$ be a homogenous ideal such that the
initial ideal $in(I)$ with respect to the rev-lex order is
weakly stable, then $reg(I)=reg(in(I))$.
\end{thm}

These two statements allow us to compute the regularity of an
ideal, if the initial ideal is weakly stable. Of course the easiest variants of Caviglia's example are ideals of the form
$I_{ij}=(x_{1}^d,x_{2}^d,x_{1}^ix_{3}^{d-i}-x_2^jx_4^{d-j})$. It
is straightforward to see that $in(I_{ij})$ is weakly stable and
that regularity of this ideal is large. For example
$reg(I_{1j})=d^2-1$ and if  $i=j=2$, we
have:
\[reg(I_{22})= \left\{
  \begin{array}{l l}
    \frac{(d+2)(d-1)}{2} & \quad \text{if $d$ is odd}\\
    \frac{d^2-2}{2} & \quad \text{if $d$ even}
  \end{array} \right.\]
  \

Note that $I_{22}$ has weaker regularity than $I=I_{11}$ which is
Caviglia's example. The proofs of the above claims can be seen by the same method
as in the following.\\

In \cite{Ch}, the regularity of powers of $I$ is computed using local
cohomological arguments. Here we compute the regularity of $I^2$
by computing its initial ideal. In particular, one sees that $I^2$
is also an ideal with large regularity.

\begin{thm} Let $I=(x_1^d,x_2^d,x_1x_3^{d-1}-x_2x_4^{d-1})$. Consider the ideal $J=I^2$ given by
\[
(x_{1}^{2d},x_{2}^{2d},x_{1}^{d}x_{2}^{d},x_{1}^{d}(x_{1}x_{3}^{d-1}-x_{2}x_{4}^{d-1}),x_{2}^{d}(x_{1}x_{3}^{d-1}-x_{2}x_{4}^{d-1}), (x_{1}x_{3}^{d-1}-x_{2}x_{4}^{d-1})^{2})
\]
in $k[x_1,...,x_4]$. Then it holds that: $reg(J)=d^2+d-1$ for all
$d\geq 2$.
\end{thm}

\begin{proof} We wield the Buchberger's algorithm to find a
Gr\"obner basis for the ideal $J$ which yields the following set
of generators for $in(J)$:
\begin{align*}
\begin{split}
&\{x_{1}^{2d},x_{2}^{2d},x_{1}^{d}x_{2}^{d},x_{1}^{d+1}x_{3}^{d-1},x_{1}x_{2}^{d}x_{3}^{d-1},
x_{1}^{2}x_{3}^{2(d-1)}\} \cup\\
&\{x_1^{d-i}x_2^{d+i}x_4^{i(d-1)}|i=1,...,d-1\} \cup\\
&\{x_1^{d-i}x_2^{i+1}x_3^{d-1}x_4^{(i+1)(d-1)}|i=0,...,d-2\} \cup\\
&\{x_1^{2d-i}x_2^{i}x_4^{i(d-1)}|i=1,...,d-1\}
\end{split}
\end{align*}

Note that unlike the case for $I$, it is not the case here that
all S-polynomials are reduced, nor is it true that the
S-polynomials are all monomials. For the sake of completeness we
present the computation of $in(J)$ in what follows. This also has
the advantage of showing our general method in the later results.
Set
\begin{align*}
\begin{split}
&g_1=(x_{1}x_{3}^{d-1}-x_{2}x_{4}^{d-1})^{2},\\
&g_2=x_{2}^{d}(x_{1}x_{3}^{d-1}-x_{2}x_{4}^{d-1}),\\
&g_3=x_{1}^{d}(x_{1}x_{3}^{d-1}-x_{2}x_{4}^{d-1}),\\
&g_4=x_{1}^{d}x_{2}^{d},\\
&g_5=x_{2}^{2d},\\
&g_6=x_{1}^{2d},\\
&G=\{g_1,g_2,g_3,g_4,g_5,g_6\}.
\end{split}
\end{align*}
	
Let us compute the S-polynomials to find a Gr\"obner basis. One sets
$H_7:=S(g_1,g_3)=-x_1^dx_2x_3^{d-1}x_4^{d-1}+x_1^{d-1}x_2^2x_4^{2(d-1)}$
which gives rise to $g^{\prime}_7=x_1^dx_2x_3^{d-1}x_4^{d-1}$. For $i=1,...,d-2$, set recursively
\[
H_{7+i}=S(g_1,H_{6+i})=-x_1^{d-i}x_2^{i+1}x_3^{d-1}x_4^{(i+1)(d-1)}+x_1^{(d-i-1)}x_2^{(i+2)}x_4^{(i+2)(d-1)}
\]
which yields the generator
$g^{\prime}_{7+i}=x_1^{d-i}x_2^{i+1}x_3^{d-1}x_4^{(i+1)(d-1)}$ for
$i=0,...,d-2$. Note that for $i=d-1$,
$H_{7+(d-1)}=H_{6+d}\xrightarrow{G}0$, i.e., $H_{6+d}$ reduces to zero with respect to the set $G$ and so this sequence stops.
We compute further that $g^{\prime
\prime}_7:=S(g_2,g_4)=x_1^{d-1}x_2^{d+1}x_4^{d-1}$ and for
$i=2,...,d-1$, we set: $g^{\prime \prime}_{6+i}:=S(g_2,g^{\prime
\prime}_{5+i})=x_1^{d-i}x_2^{d+i}x_4^{i(d-1)}$.
Note that for $i=d$, $g^{\prime \prime}_{6+d}=S(g_2,g^{\prime
\prime}_{5+d})\xrightarrow{G} 0$. Finally note that: $g^{\prime \prime
\prime}_7=S(g_3,g_6)=x_1^{2d-1}x_2x_4^{d-1}$ and for $i=2,...,d-1$, one sets
$g^{\prime \prime \prime}_{6+i}=S(g_3,g^{\prime \prime
\prime}_{5+i})=x_1^{2d-i}x_2^ix_4^{i(d-1)}$. We remark that for $i=d$, $g^{\prime \prime \prime}_{6+d}=S(g_3,g_{5+d})\xrightarrow{G} 0$. By adding these new generators to $G$ one sees that all other
$S$-polynomials can be reduced to zero with respect to the new set
and hence we have found a Gr\"obner Basis for $J$ and  all of the
generators of $in(J)$ are as above. We order the generators with respect to the reverse lexicogeraphical order as follows:
\[
x_{1}^{2d}<...<x_1^{2}x_2^{d-1}x_3^{d-1}x_4^{(d-1)^2}
\]
Note that $in(J)$ is a weakly stable ideal and therefore its
regularity is equal to the maximum of the numbers
$deg(u_i)+C(u_i)$. This maximum is obtained at the last generator
and $C(u_i)$ is given by $x_3^{d-2}$ and
therefore $\reg(J)=(d-1)^2+2d+(d-2)=d^2+d-1$.
\end{proof}

\begin{thm} In $S=k[x_1,x_2,x_3,x_4,x_5]$, let $I$ be the ideal:\\

$I=(x_{1}^{d},x_{2}^{d},x_{3}^{d},x_{1}x_{2}^{d-1},x_{1}x_{3}^{d-1},x_{2}x_{4}^{d-1}-x_{3}x_{5}^{d-1})$\\

Then $\reg(I)=d^2-1$ for $d\geq 2$.\\

\end{thm}

\begin{proof} By the same method as above, one can show that\\

$\ini(I)=\{x_{1}^{d},x_{2}^{d},x_{3}^{d},x_{1}x_{2}^{d-1},x_{1}x_{3}^{d-1},x_{2}x_{4}^{d-1}\}
\cup \{x_1x_2^{d-(i+1)}(x_3x_5^{d-1})^i|i=1,...,d-2\}\cup
\{x_2^{d-i}(x_3x_5^{d-1})^i|i=1,...,d-1\}$.\\

This ideal is weakly stable. We order the generators as follows\\

$x_{1}^{d}<...<x_2(x_3x_5^{d-1})^{d-1}$\\

The maximum of the degrees is $1+d(d-1)$ and $C(u_i)$ can be given
by $x_4^{d-2}$. It follows that: $\reg(I)=1+d(d-1)+d-2=d^2-1$.
\end{proof}

In what follows, we give an example of an ideal in the polynomial
ring $S=k[x_1,x_2,x_3,x_4,x_5,x_6]$ with six generators whose
regularity is a polynomial of degree $3$ in the generating degree
$d$ of the ideal. \\

\begin{thm} In $S=k[x_1,x_2,x_3,x_4,x_5,x_6]$ consider the ideal
\[
I=(x_{1}^{d},x_{2}^{d},x_{3}^{d},x_{4}^{d},x_{1}x_{3}^{d-1}-x_{2}x_{4}^{d-1},x_{3}x_{5}^{d-1}-x_{4}x_{6}^{d-1}),
\]
then $\reg(I)=d(d-1)(d-2)+3d-3=d^3-3d^2+5d-3$.
\end{thm}

\begin{proof}
One could show that $\ini(I)$ is generated by the following set:
\begin{align*}
\begin{split}
&\{x_{1}^{d},x_{2}^{d},x_{3}^{d},x_{4}^{d},x_{1}x_{3}^{d-1},x_{3}x_{5}^{d-1},
x_1^{d-1}x_2x_4^{d-1},x_2x_3x_4^{d-1},x_{2}x_{4}^{d-1}x_{5}^{d-1}\}\cup\\ 
&\{x_1^{j}x_2^{d-j}x_3^{d-2-i}x_4^{i+1}x_6^{((j-1)d-(j-2)+i)(d-1)}| i=0,..,d-2,j=1,...,d-2\}\cup\\
&\{x_1^{d-1}x_2x_3^{d-2-i}x_4^{i+1}x_6^{((d-2)d-(d-3)+i)(d-1)}| i=0,...,d-3\}\cup\\
&\{x_1x_3^{d-i}x_4^{i}x_6^{(i-1)(d-1)}|i=2,...,d-1\}\cup\\
&\{x_3^{d-i}x_4^{i}x_6^{i(d-1)}| i=1,...,d-1\}
\end{split}
\end{align*}

This ideal is weakly stable. Moreover, one sees that the maximum
of the numbers in Theorem 2.2 is obtained at the generator
$x_1^{d-1}x_2x_3x_4^{d-2}x_6^{d(d-2)(d-1)}$ and an example of
$C(u_j)$ could be: $x_5^{d-2}$. This means that:
$\reg(I)=d(d-2)(d-1)+d-2+2+d-1+d-2=d(d-1)(d-2)+3d-3=d^3-3d^2+5d-3$.
\end{proof}

\begin{rem} Presumably by a similar argument as above one
can show that in the ring $S=k[x_1,...,x_{2n}]$ the ideal given by
$I=(x_1^d,...,x_{2n-2}^d)+(x_{2i+1}x_{2i+3}^{d-1}-x_{2i+2}x_{2i+4}^{d-1}\; |\; 0\leq i\leq n-2)$ is a polynomial of degree $n$. We expect that
$in(I)$is weakly stable. Note that for this ideal, $in(I)$ is very
plausible to be weakly stable as the ideal already contains pure
powers of $x_1,...,x_{2n-2}$. That is, pure powers of all of
the variables except two. For example for $n=4$, then
\begin{align*}
\begin{split}
I=(&x_{1}^{d}, x_{2}^{d}, x_{3}^{d}, x_{4}^{d}, x_{5}^{d}, x_{6}^{d},\\
     &x_{1}x_{3}^{d-1}-x_{2}x_{4}^{d-1}, x_{3}x_{5}^{d-1}-x_{4}x_{6}^{d-1},x_{5}x_{7}^{d-1}-x_{6}x_{8}^{d-1})\\
\end{split}
\end{align*}
and one could show that
$in(I)$ is weakly stable and \[\reg(I)=(d-2)(d(d-1)^2+3)+3\]. In order to do some computational experiments, let us consider this example in the following simple code in CoCoA to compute the \textbf{BettiDiagram} and hence the regulrity of $I$ when $d=3,\cdots,10$ for example.

\begin{flushleft}
{\small\tt\boldmath
Use R ::= Q[x[1..8]];\\
For D:=3 To 10 Do\\
$\qquad$ I := Ideal($x[1]^D, x[2]^D, x[3]^D, x[4]^D, x[5]^D, x[6]^D,$\\
$\qquad\qquad\qquad\qquad x[1]x[3]^{(D-1)}-x[2]x[4]^{(D-1)},$\\
$\qquad\qquad\qquad\qquad x[3]x[5]^{(D-1)}-x[4]x[6]^{(D-1)},$\\
$\qquad\qquad\qquad\qquad x[5]x[7]^{(D-1)}-x[6]x[8]^{(D-1)}$);\\
$\qquad$BettiDiagram(I);\\
EndFor;\\
}
\end{flushleft}
and here are the regularities:
\begin{align*}
\begin{split}
&\text{If }d=3 \text{ then } reg(I)=18\\
&\text{If }d=4 \text{ then } reg(I)=81\\
&\text{If }d=5 \text{ then } reg(I)=252\\
&\text{If }d=6 \text{ then } reg(I)=615\\
&\text{If }d=7 \text{ then } reg(I)=1278\\
&\text{If }d=8 \text{ then } reg(I)=2373\\
&\text{If }d=9 \text{ then } reg(I)=4056\\
&\text{If }d=10 \text{ then } reg(I)=6507\\
\end{split}
\end{align*}
and all of them satisfy the presented formula.
\end{rem}

\begin{rem} It can be seen that the large regularity of
the ideal in Theorem 2.5 is revealed in the first syzygy of the
ideal. In fact if one defines $t_i(I)=\max
\{j|\beta_{i,j}(S/I)\neq 0\}$, where the $\beta_{i,j}$ are the
Betti numbers of $S/I$. Then it follows that $t_1(S/I)=d$ and
$t_2(S/I)=reg(I)$; see \cite{M}.
\end{rem}

\section{Ideals with regularity jumps}

In this section we give several examples of the jump phenomenon introduced
in \cite{Co}. The notion of regularity jump
has been defined in \cite{Co} as follows:\\

\begin{defn} An equigenerated ideal $I$ in degree $d$
is said to have regularity jump at $k$ (or that the regularity of
powers of $I$ jumps at place $k$)
if $reg(I^k)-reg(I^{k-1})>d$.
\end{defn}

The first example of such an ideal was given by Terai.
\begin{exam} (Terai)
This ideal is
\[I=(x_1x_2x_3,x_1x_2x_4,x_1x_3x_5,x_1x_4x_6,x_1x_5x_6,x_2x_3x_6,x_2x_4x_5,x_2x_5x_6,x_3x_4x_5,x_3x_4x_6)\]
for which $reg(I)=3$ and $reg(I^2)=7$. The non-linear syzygy of
$I^2$ appears at the end of the resolution.
\end{exam}

The example that we consider is a generalization of example 2.10
in \cite{Co}. In the aforementioned paper \cite{Co} it is guessed that this
family of ideals has regularity jumps at $k=2$ by declaring that
there are some experimental evidences that $I^2_n$ has non-linear
syzygy. Here we take a different approach and prove that
specific graded pieces of the graded local cohomology modules do
not vanish, leading to the fact that $reg(I^2_n)>4$. Note that ideals $I_n$ define Cohen-Macaulay rings of minimal multiplicity and our result shows that even among such ideals one can find infinitely many examples whose squares do not have linear resolution. The example is as follows.

\begin{exam} Let
\begin{align*}
\begin{split}
I_n=(&x_1^2,...,x_{n+1}^2,x_1x_2,...,x_1x_{n+1},\\
&x_2x_3-x_1x_{n+2},x_2x_4-x_1x_{n+3},...,x_2x_{n+1}-x_1x_{2n},...,\\
&x_3x_4-x_1x_{2n+1},...,x_3x_{n+1}-x_1x_{3n-2},...,x_nx_{n+1}-x_1x_{s}),
\end{split}
\end{align*}
where $s=\frac{n(n+1)}{2}+1$.
\end{exam}

In this description, a typical generator apart from
$x_1^2,...,x_{n+1}^2,x_1x_2,...,x_1x_{n+1}$ is of the form
$x_ix_j-x_1x_{t(n,i,j)}$, where $2\leq i<j\leq n+1$ and
$t(n,i,j)=(i-1)n-\frac{(i-1)(i-2)}{2}+1+(j-i)$

Another description of this ideal as given in \cite{Co} is as follows:
\[
I_n=(x^2,y_1^2,y_2^2,...,y_n^2,xy_1,...,xy_n,y_iy_j-xz_{i,j})\text{ for } 1\leq i<j\leq n.
\]

Then it holds that:
\begin{thm}
For $I_n$ as above, $reg(I_n)=2$ and $reg(I^2_n)>4$. Therefore we
get an infinite family of ideals with regularity jumps at $k=2$.
\end{thm}

\begin{proof} Although the first description seems more complicated than the
second one, we prefer, in order to avoid complicated indices, to
work with the first description. It is straightforward to check that\\
$in(I_n)=(x_1^2,...,x_{n+1}^2,x_1x_2,...,x_1x_{n+1},x_2x_3,x_2x_4,...,x_nx_{n+1})$.
In fact the $S$-polynomials are all reduced. It follows that
$in(I_n)$ is weakly stable and therefore
regularity of $I_n$ and $in(I_n)$ are both equal to the maximum
degree of generators of $in(I_n)$, i.e., equal to $2$. Note however that
$in(I_n^2)$ is \emph{not} weakly stable.\\

In order to show that $reg(I^2_n)>4$, setting $J_n=I^2_n$, we show
that there always exists an integer $l$ such that
$H^{l}_{m}(S/J_n)_{4-l}\neq 0$. By Remark 1.1 it follows that
$reg(J_n)>4$. To this end, we use the fact that local cohomology
can be computed via \v{C}ech complex. That is, the following
complex:\\

$0\to S/J_n \xrightarrow{d^0} \oplus
(S/J_n)_{x_i}\xrightarrow{d^1} \oplus (S/J_n)_{x_ix_j}\to ....
\xrightarrow{d^{n-1}} (S/J_n)_{x_1...x_s}\to 0$\\

Where the maps are alternating sums of localization maps. See
\cite{BS}, \S 5.1. Note that since $x_1^4=...=x_{n+1}^4=0$ in $S/J_n$,
the only localized summands that contribute to the above complex
are localizations at $x_j$ with $j=n+2,...,s$. In other words, the
cohomology can be computed
by the complex\\

$0\to S/J_n \xrightarrow{d^0} \oplus_{n+2\leq i\leq s}
(S/J_n)_{x_i}\xrightarrow{d^1} \oplus_{n+2\leq i<j \leq s}
(S/J_n)_{x_ix_j}\to ....
\xrightarrow{d^{s-n-1}} (S/J_n)_{x_{n+2}...x_{s}}\to 0$\\

We first describe our method for the simplest case of $n=3$ and
then write down the natural generalization.\\

Let $n=3$. Then in $S/J_3$, we have the following equalities:
(note that we abuse the notation and show the image of an element
$\beta \in S$ in $S/J$ again by $\beta$)
\begin{align*}
\begin{split}
&x_1x_4(x_2x_3-x_1x_5)=0\Rightarrow x_1x_2x_3x_4=x_1^2x_4x_5\\
&x_1x_2(x_3x_4-x_1x_7)=0\Rightarrow x_1x_2x_3x_4=x_1^2x_2x_7\\
&x_1x_3(x_2x_4-x_1x_6)=0\Rightarrow x_1x_2x_3x_4=x_1^2x_3x_6\\
\end{split}
\end{align*}

and also:\\

$x_1x_2(x_2x_3-x_1x_5)=0\Rightarrow x_1^2x_2x_5=0$, because
$x_1x_2^2x_3=(x_1x_3)(x_2^2)=0$.\\

Similarly,
$x_1^2x_3x_5=x_1^2x_2x_6=x_1^2x_4x_6=x_1^2x_3x_7=x_1^2x_4x_7=0$\\

It follows, combining the above sets of equalities, that the
element $\alpha:=x_1x_2x_3x_4$ is annihilated by $x_5$, $x_6$ and
$x_7$. This shows that $\alpha\in ker(d^0)_4=H^{0}_{m}(S/J_3)_{4}$
and hence $H^{0}_{m}(S/J_3)_{4}\neq 0$. \\

The above ideas can be generalized to arbitrary $n$. In fact, in
the general case, we have the following in $S/J_n$:
\[
x_1x_2x_3x_4=x_1^2x_4x_{n+2}=x_1^2x_3x_{n+3}=x_1^2x_2x_{2n+1}
\]

If $x_ix_j-x_1x_r$ is a generator of $I_n$ such that $\{i,j\}\cap
\{2,3,4\}\neq \emptyset$, then by the same argument as in the
$n=3$
case, it follows that
\[
x_1^2x_tx_r=0\text{ for all }t\in \{i,j\}\cap \{2,3,4\}.
\]

Now let $x_{i_1}x_{j_1}-x_1x_{r_1}$,...,
$x_{i_l}x_{j_l}-x_1x_{r_l}$ be the set of all generators of $I_n$
of the form $x_ix_j-x_1x_r$ such that $\{i,j\}\cap \{2,3,4\}=
\emptyset$. Set $\alpha:=x_1x_2x_3x_4$ as before. Then the above
equalities show that for
\[
j\in \{n+2,...,s\}\setminus \{r_1,...,r_l\}, \alpha x_j=0.
\]
This implies that the element
$\kappa:=(0,...,\frac{\alpha}{x_{r_1}...x_{r_l}},...,0)\in
C^{l}(S/J_n)$ lies in $ker(d^l)_{4-l}$. It follows that
$\kappa\in H^{l}_{m}(S/J_n)_{4-l}$
which is non-zero in this
cohomology module and hence $H^{l}_{m}(S/J_n)_{4-l}\neq 0$.
\end{proof}

\section*{Acknowledgment} The authors gratefully acknowledge the computer algebra system CoCoA in which the presented examples and the performed computations have been tested. They actually designed the first motivations for this paper.


\begin{thebibliography}{1}

\bibitem{B}K.Borna, \emph{On linear resolution of powers of an ideal}, Osaka J. Math., \textbf{46} (2009), no. 4, ~1047--1058.

\bibitem{BM}D. Bayer, D. Mumford, \emph{What can be computed in algebraic
geometry?,} In: Computational algebraic geometry and commutative
algebra (Cortona, 1991), Sympos. Math.,XXXIV, Cambridge Univ.
Press, Cambridge, pp 1-48.

\bibitem{BMNSSS}J. Beder, J. McCullough, L. Nunez-Betancourt,  A. Seceleanu, B. Snapp, B.
Stone, \emph{Ideals with larger projective dimension and
regularity,} J. Sym. Comp. 46 (2011), 1105-1113.

\bibitem{BS}M. Brodmann, R. Sharp, \emph{Local cohomology: an algebraic introduction with geometric
applications,} Cambridge Studies in Advanced Mathematics, 60.
Cambridge University Press, Cambridge, 1998.

\bibitem{C}G. Caviglia, \emph{Kozul algebras, castelnuovo-mumford regularity and generic initial
ideals,} Ph.D. thesis, University of Kansas, 2004.

\bibitem{Ch}M. Chardin, \emph{Some results and questions on Castelnuovo-Mumford
regularity,} Syzygies and Hilbert Functions. Lecture Notes in Pure
and Appl. Math. 254, 1-40, 2007.

\bibitem{Co}A. Conca, \emph{Regularity jumps for powers of ideals,} in: Commutative Algebra: Geometric,
Homological Combinatorial and Computational Aspects,(A. Corso, P.
Gimenez, M. Vaz Pinto and S. Zarzuela, Eds.), Lecture Notes in
Pure and Applied Mathematics 244, Chapman Hall/CRC, Boca Roton,
2006.

\bibitem{E}D. Eisenbud, \emph{The Geometry of Syzygies,} Graduate Texts in Mathematics, vol. 229, Springer, 2005.

\bibitem{EG}D. Eisenbud, S. Goto, \emph{Linear free resolutions and minimal multiplicity,} J. Algebra 88 (1984), 89-133.

\bibitem{L}R. Lazarsfeld, \emph{Positivity in algebraic geometry,
I: classical setting: line bundles and linear series}. Ergebnisse
der Math. (3) 48, Springer, Berlin, 2004. MR 2005k: 14001a.

\bibitem{M}J. McCullough, \emph{A Polynomial Bound on the Regularity of an Ideal in Terms of Half of the
Syzygies,} Math. Res. Ltrs. 19 (2012), no 3, 555-565.

\bibitem{MM}E. W. Mayr, A.R. Meyer, \emph{The complexity of the word problems for commutative
semigroups and polynomial ideals,} Adv. in Math. 46 (1982), no. 3,
305-329.

\end{thebibliography}
\end{document}